\documentclass[reqno,12pt]{amsart}
\usepackage{amsmath,amssymb}

\newtheorem{theorem}{Theorem}[section]
\newtheorem{proposition}[theorem]{Proposition}

\newtheorem{lemma}[theorem]{Lemma}
\newtheorem{corollary}[theorem]{Corollary}
\newtheorem{remark}[theorem]{Remark}

\numberwithin{equation}{section}
\numberwithin{theorem}{section}

\newcommand{\mc}[1]{{\mathcal #1}}

\newcommand{\bb}[1]{{\mathbb #1}}
\newcommand{\eps}{\varepsilon}

\newcommand{\Glimsup}{\mathop{\textrm{$\Gamma\!$--$\varlimsup$}}\limits}
\newcommand{\Gliminf}{\mathop{\textrm{$\Gamma\!$--$\varliminf$}}\limits}

\newfont{\indic}{bbmss12}
\def\un#1{\hbox{{\indic 1}$_{#1}$}}

\begin{document}

\title[Fractional conservation laws]
{A variational approach to the
  inviscous limit of fractional conservation laws}

\author[M.\ Mariani]{Mauro Mariani} %
\address{Mauro Mariani, Laboratoire d'Analyse, Topologie,
  Probabilit\'es (CNRS UMR 6632), Universit\'e Aix-Marseille 3,
  Facult\'e des Sciences et Techniques Saint-J\'er\^ome, Avenue
  Escadrille Normandie-Niemen 13397 Marseille Cedex 20, France.} %
\email{mariani@cmi.univ-mrs.fr} %
\author[Y.\ Sire]{Yannick Sire} %
\address{Yannick Sire, Laboratoire d'Analyse, Topologie,
  Probabilit\'es (CNRS UMR 6632), Universit\'e Aix-Marseille 3,
  Facult\'e des Sciences et Techniques Saint-J\'er\^ome, Avenue
  Escadrille Normandie-Niemen 13397 Marseille Cedex 20, France.} %
\email{sire@cmi.univ-mrs.fr} %

\begin{abstract}
  We are concerned with a control problem related to the vanishing
  \emph{fractional} viscosity approximation to scalar conservation
  laws. We investigate the $\Gamma$-convergence of the control cost
  functional, as the viscosity coefficient tends to zero.
 \end{abstract}
\maketitle

\section{Introduction}
\label{s:1}
\subsection{Optimal control for conservation laws}
We are concerned with the scalar one-dimensional conservation law
\begin{equation}
  \label{e:1.1}
   \partial_t u+\partial_x f(u)=0 
\end{equation}
where the time variable $t$ runs on a given interval $[0,T]$, the
space variable $x$ runs, for the sake of simplicity, on a one
dimensional torus $\bb T$, and $u= u(t,x)$. Even if the initial datum
$u(0)= u(0,\cdot)$ is smooth, the flow \eqref{e:1.1} may develop
singularities, so that in general no classical smooth solutions
exist. On the other hand, if $f$ is nonlinear, there are in general
infinitely many weak solutions to the Cauchy problem associated to
\eqref{e:1.1}. Existence and uniqueness of the solution are then
recovered by imposing the so-called entropy condition. A celebrated
result by Kruzhkov states the uniqueness of the entropy solution to
the Cauchy problem associated to \eqref{e:1.1}. Such an entropic
solution can be obtained as limit of various approximations of the
flow \eqref{e:1.1}; namely the entropy solution is the relevant
one. We refer to \cite{Da,Pe,Se} for general theory of conservation
laws.

In particular, see \cite{DI}, the entropy solution to \eqref{e:1.1}
can be recovered as the limit as $\eps \to 0$ of
solutions to
\begin{equation}
  \label{e:1.2}
   u_t+\partial_x f(u)= - \frac{\eps}2 (-\partial_{xx})^s u
\end{equation}
where $1/2 < s \le 1$, $(-\partial_{xx})^s$ denotes the $s$-th power
of the negative Laplacian. In this paper a more general variational
approach to the above problem will be addressed. Indeed \eqref{e:1.2}
is a model for nonlinear transport-diffusion phenomena in a media
allowing long-range correlations. The action of an external field $E$
on the system modifies \eqref{e:1.2} to
\begin{equation}
  \label{e:1.3}
   \partial_t u + \partial_x f(u)=- \frac{\eps}2 (-\partial_{xx})^s u +
   (-\partial_{xx})^{s/2} E
\end{equation}
while the work done by the external field $E$ equals $\frac 12
\|E\|_{L_{2}}^{2}$.  A control problem is then naturally introduced by
defining, for $\eps>0$, the functional (see Section~\ref{s:2} for a
more precise definition)
\begin{equation}
\label{e:I1}
I_{\eps}(u):= \inf_{E}  \frac{\eps^{-1}}2 \|E\|_{L_{2}(dt,dx)}^{2}
\end{equation}
where the infimum is carried over the fields $E$ such that
\eqref{e:1.3} holds. We then investigate the variational convergence
of $I_\eps$.

\subsection{A Statistical Mechanics interpretation}
Roughly speaking, \eqref{e:1.2} can be interpreted as the
\emph{typical} evolution behavior of a density (e.g.\ a density of
charge) in a media with long-range correlations. Then \eqref{e:1.3}
describes the same density when a random fluctuation $E$ is
introduced, while $\|E\|_{L_2}^2$ gives a weight to how much unlikely
the fluctuation is. Thus, $I_\eps(u)$ quantifies how unlikely is to
observe a density $u$, when the typical behavior is given by
\eqref{e:1.2}. By the end, one may interpret $I_\eps$ as a free energy
of the system, and thus one would be interested to understand the
typical behavior of the infima of $I_\eps$ over good sets, as $\eps
\to 0$.

Indeed the $\eps \to 0$ limit corresponds, in this Statistical
Mechanics' description, to a hydrodynamical limit, so that the
variational limit of $I_\eps$ should play the role of a free energy
for the limiting (macroscopic) system. While this picture has been
investigated in several models, see \cite{Spohn}, and connections of
microscopic, mesoscopic and macroscopic descriptions have been
sometimes established rigorously, most of the literature concerns
diffusive systems, where the limiting macroscopic behavior is of
parabolic type. A major open problem concerns fluctuations of systems
with an hydrodynamics given by non-linear transport evolutions as
\eqref{e:1.1}. No result at all concerning nonlocal (that is,
long-range) fluctuations of hyperbolic systems is known to the
authors.

It is well known that, in order to investigate the limit of infima of
a sequence $I_\eps$ of functional, the notion of $\Gamma$-convergence,
see Section~\ref{ss:gc}, is the relevant one. In \cite{BBMN} the
$\Gamma$-convergence of $I_\eps$ is investigated in the "local" case
$s=1$ corresponding to short-range correlations. In \cite{Mauro} such
results are rigorously connected to the description of large
deviations for stochastic PDEs modeling stochastic particles
systems. In both papers, two different scalings are considered,
corresponding to the $\Gamma$-limits of $\eps I_\eps$ and of
$I_\eps$. In this paper we only address the latter problem, the
$\Gamma$-limit of $I_\eps$. Indeed, only in this latter scaling the
vanishing diffusive term $\eps (-\partial_{xx})^s u$ is expected to play a different role than the standard Laplacian.

The main results here established are the following. A functional $I
\colon L_p([0,T]\times \bb T) \to [0,+\infty]$ is introduced, see
\eqref{e:Ie}. $I(u)$ is set to be $+\infty$ if $u$ is not a weak
solution to \eqref{e:1.1}, while if $u$ solves \eqref{e:1.1}, $I(u)$
quantifies how much the entropic condition of \eqref{e:1.1} is
violated by $u$. In Theorem~\ref{t:ecne}-(i) we prove that if $u_\eps
\to u$ in $L_p$, then $\varliminf I_\eps(u_\eps)\ge I(u)$, a so-called
$\Gamma$-liminf inequality. In Theorem~\ref{t:ecne}-(ii), we prove
that $I_\eps$ is an equicoercive sequence in the strong $L_p$
topology. The two results imply that if $\mc C$ is closed in $L_p$,
then $$\varliminf_\eps \inf_{u \in \mc C} I_\eps(u)\ge \inf_{u \in \mc C} I(u).$$In order
to characterize the $\Gamma$-limit of $I_\eps$, one would need to
establish a $\Gamma$-limsup inequality. This step is missing even in
the local case $s=1$, mainly because of open issues concerning
chain-rules for non BV fields, see \cite{BBMN,DOW}. We thus do not
tackle the problem here, but rather give a qualitative hint that may
suggest $I$ to be the $\Gamma$-limit (and not just an upper bound of
the $\Gamma$-limit ) of $I_\eps$. Indeed, in Theorem~\ref{t:quasipot}
we explicitly calculate the quasipotential of $I_\eps$, a proper way
to describe the long time asymptotic of a functional, and prove it to
be independent of $\eps$ and equal to the quasipotential of $I$. To
put it shortly, the $\eps \to 0$ limit and $T \to +\infty$ limit
commute, if one assumes $I$ to be the $\Gamma$-limit of $I_\eps$.

The functional $I$ was already introduced in \cite{BBMN}. The result
in this paper then asserts that the limiting fluctuations are the same
if $1/2<s<1$ or $s=1$. $I$ thus appears to be a solid candidate as
the proper generalization of the functional introduced by Jensen and
Varadhan in a stochastic particles setting (see e.g.\ \cite{Ragu} for
a summary of their results).

\subsection{Outline and generalizations} 
Beyond considering the non-local case $s<1$, two further technical difficulties are addressed in this paper with respect to \cite{BBMN}. First, we allow unbounded densities $u \in L_p([0,T]\times \bb T)$ (while $u$ was a priori restricted to take values in $[0,1]$ in \cite{BBMN}), and we fix an initial datum $u_0$ (whereas no initial condition was given in \cite{BBMN}). However, since we only achieve $L_2$ H\"older a priori bounds on $u$, we need $f$ to be uniformly Lipschitz (to make integrals meaningful) and $p<2$ (to assure a needed uniform integrability in the topology considered).

From a technical point of view, the key proofs are achieved by heavily using the Caffarelli-Silvestre \cite{CS} representation of the fractional Laplacian operator, as opposed on the $s=1$ case where of course only local evaluations were needed.

The paper is organized as
follows. In Section~\ref{s:2} the main results are stated. In
Section~\ref{s:3} some useful properties of the fractional Laplacian
are recalled. In Section~\ref{s:4} we establish the basic estimates
needed in Section~\ref{s:5}, where the main results about
$\Gamma$-convergence are proved. Section~\ref{s:6} is devoted to the
proof of Theorem~\ref{t:quasipot}, characterizing the quasipotential.

We remark that we tried to keep the setting as readable as
possible. Some generalizations are possible by using the same
techniques of the paper. First, one can flawlessly change the torus
$\bb T$ with the real line $\bb R$. Secondly, one may study the
problem in higher dimensions: all the proofs go through but
Theorem~\ref{t:ecne}-(ii) which needs to be addressed by the means of
averaging lemmas in this case, see \cite[Chap.~5]{Pe}, and thus requires
stronger hypotheses on $f$.

\section{Preliminaries and main results}
\label{s:2}
Let $\bb T=\bb R/\bb Z$ be the one dimensional torus and let $\langle
\cdot,\cdot \rangle$ denote duality in $L_2(\bb T)$. Hereafter $T>0$
is fixed, $\partial_t$ denotes derivative with respect to the time
variable $t\in [0,T]$, $\partial_x$ derivative with respect to the
space variable $x\in \bb T$.

Let $C\big([0,T]; H^{-1}(\bb T)\big)$ be endowed with
its natural metric
\begin{equation}
\label{e:2.2}
  \mathrm{distance}(u,v):= 
  \sup \big\{ 
              \langle u(t)-v(t), \varphi \rangle,\, 
              t\in [0,T],\,\|\varphi\|_{L_2(\bb T)}^2 
               + \|\partial_x \varphi\|_{L_2(\bb T)}^2 \le 1 
         \big\}
\end{equation}
Fix once and for all $p \in [1,2[$, and let $\mc X:= C\big([0,T];
H^{-1}(\bb T)\big) \cap L_p([0,T]\times \bb T)$ be endowed with the
refinement of the $C\big([0,T]; H^{-1}(\bb T)\big)$ and the
\emph{strong} $L_p$ metrics.

Let moreover $H_s(\bb T)$, $\dot H_s(\bb T)$ be the fractional Sobolev
space and the homogeneous fractional Sobolev space of exponent
$s>0$. $H_{-s}(\bb T)$ and $\dot H_{-s}(\bb T)$ denotes their dual
spaces. Notice that with this notation $ \dot H_{-s} = \big\{
(-\partial_{xx})^s h,\,h\in \dot H_s(\bb T)\big\}$. We finally
introduce the space $\mc H= L_2\big([0,T]; \dot H_{s}(\bb T)\big)$ and
its dual $\mc H^*= L_2\big([0,T]; \dot H_{-s}(\bb T)\big)$. We use the
standard notation for the norms, for instance $\|g\|_{\mc
  H^*}^2=\int_0^T \big|(-\partial_{xx})^{s/2}h(t)\big|^2\,dt$ if
$g(t)=(-\partial_{xx})^s h(t)$ for a.e.\ $t$, while $\|g\|_{\mc
  H^*}=+\infty$ if $g \not \in \mc H^*$.

\subsection{Fractional parabolic cost functional}
We assume the flux $f$ to be bounded and Lipschitz, the initial datum
$u_0 \in L_2(\bb T)$, and the exponent $s>1/2$. For $\eps>0$ the
functional $I_\eps \colon \mc X \to [0,+\infty]$ is defined as
\begin{equation}
\label{e:Ie}
I_\eps(u):=
\begin{cases}
\frac{\eps^{-1}}2 \big\|\partial_t u+ \partial_x f(u) 
              +\frac{\eps}2 (-\partial_{xx})^s u \big\|_{\mc H^*}^2
& \text{if $u \in \mc H \cap C([0,T];L_2(\bb T))$}
\\
& \text{and $u(0,x)=u_{0}(x)$}
\\
+\infty & \text{otherwise}
\end{cases}
\end{equation}

The following proposition provides a characterization of $I_\eps$ as
the cost functional of the optimal control problem introduced in
Section~\ref{s:1}.
\begin{remark}
\label{r:riesz}
If $u\in \mc X$ is such that $I_{\eps}(u)<\infty$, then there exists a
unique $\Phi \equiv \Phi_u \in \mc H$ such that the equation
\begin{equation}
\label{e:control}
\partial_t u + \partial_x f(u)+\frac{\eps}2 (-\partial_{xx})^s u 
    = (-\partial_{xx})^s \Phi
\end{equation}
holds weakly, when checked against test functions in
$C^{\infty}([0,T]\times \bb T)$. Moreover
\begin{equation}
\label{e:Irep}
I_\eps(u):= 
     \frac{\eps^{-1}}2 \|\Phi\|_{\mc H}^2
\end{equation}
\end{remark}

The next proposition states that $I_\eps$ is a \emph{good} functional,
namely that its sublevel sets are compact, a standard requirement for
cost functionals. In particular it states that the condition $u\in \mc
H \cap C([0,T];L_2(\bb T))$ is the natural one to impose in the
definition of the domain of $I_\eps$, see \eqref{e:Ie}. For instance,
one would not in general have a lower-semicontinuous functional if
higher regularity would be required on $u$, while the representation
in Remark~\ref{r:riesz} would not hold for weaker regularity or indeed
if $s<1/2$.
\begin{proposition}
\label{p:semicont}
$I_\eps$ is a coercive lower-semicontinuous functional on $\mc X$.
\end{proposition}

\subsection{$\Gamma$-convergence}
\label{ss:gc}
As well known, a most useful notion of variational convergence is the
$\Gamma$-convergence which, together with some compactness estimates,
implies convergence of the minima. Recall that a sequence $(F_\eps)$
of functionals $F_\eps \colon \mc X \to [0,+\infty]$ is
\emph{equicoercive} on $\mc X$ iff for each $M >0$ there exists a
compact set $K_M$ such that $\varlimsup_{\eps \downarrow 0 } \{x \in
\mc X \,:\: F_\eps(x) \le M \} \subset K_M$. We briefly recall the
basic definitions of the $\Gamma$-convergence theory, see e.g.\
\cite{Br}. Given $x \in \mc X$ we define
\begin{eqnarray*}
\big( \Gliminf_{\eps \to 0} F_\eps \big)\, (x) &:= &
    \inf \big\{ \varliminf_{\eps \to 0} F_\eps(x^\eps),\,
                   \{x^\eps\} \subset X\,:\: x^\eps \to x \big\}
\\
\big(\Glimsup_{\eps \to 0} F_\eps \big)\, (x) &:= & 
\inf \big\{\varlimsup_{\eps \to 0}  F_\eps(x^\eps),\,
                   \{x^\eps\} \subset X\,:\: x^\eps \to x \big\}
\end{eqnarray*}
Whenever $\Gliminf_\eps F_\eps =\Glimsup_\eps F_\eps=F$ we say that
$F_\eps$ $\Gamma$-converges to $F$ in $\mc X$. Equivalently, $F_\eps$
$\Gamma$-converges to $F$ iff for each $x\in \mc X$:
\begin{itemize}
\item[\rm{--}]{for any sequence $x^\eps\to x$ we have $\varliminf_\eps
    F_\eps(x^\eps)\ge F(x)$ \ (\emph{$\Gamma$-liminf inequality});}
\item[\rm{--}]{ there exists a sequence $x^\eps\to x$ such that
    $\varlimsup_\eps F_\eps(x^\eps)\le F(x)$ \ (\emph{$\Gamma$-limsup
      inequality}).}
\end{itemize}
Equicoercivity and $\Gamma$-convergence of a sequence $(F_\eps)$ imply
an upper bound of infima over open sets, and a lower bound of infima
over closed sets, see e.g.\ \cite[Prop.~1.18]{Br}, and therefore it is
the relevant notion of variational convergence in the control setting
introduced in \eqref{e:control}.

\subsection{Solutions to scalar conservation law} 
In order to describe the candidate $\Gamma$-limit of $I_\eps$, further
preliminaries are introduced in this section.

An element $u\in \mc X$ is a \emph{weak solution} to \eqref{e:1.1}
with initial condition $u_0 \in L_2(\bb T)$ iff for each $\varphi\in
C^\infty_{\mathrm{c}}\big([0,T[\times \bb T \big)$ it satisfies
\begin{equation*}
  - \int_0^t \langle u(r),\partial_t \varphi(r) \rangle 
  - \langle f(u(r)),\partial_x \varphi(r) \rangle\,dr 
  - \langle u_0 , \varphi(0)\rangle
= 0
\end{equation*}

We denote by $C^2_{\mathrm{b}}(\bb R)$ the set of twice differentiable
functions with bounded second derivative.  A function $\eta \in
C^2_{\mathrm{b}}(\bb R)$ is called an \emph{entropy} and its
\emph{conjugated entropy flux} $q\in C^2_{\mathrm{b}}(\bb R)$ is
defined, up to an additive constant, by
\begin{equation*}
q(w):=\int^w\!dv\,\eta'(v)f'(v)
\end{equation*}
For a weak solution $u$ to \eqref{e:1.1}, for an entropy -- entropy
flux pair $(\eta,q)$, the \emph{$\eta$-entropy production} is the
distribution $\wp_{\eta,u}$ acting on
$C^\infty_{\mathrm{c}}\big(]0,T[\times\bb T\big)$ as
\begin{equation}\label{eq:entropyprod}
 \wp_{\eta,u}(\varphi):= 
    - \int_0^T \langle \eta(u(r)) , \partial_t \varphi(r)\rangle
    + \langle q(u(r)) , \partial_x \varphi(r)\, \rangle\, dr
\end{equation}

The next proposition introduces a suitable class of solutions to
\eqref{e:1.1}. Its proof is given in \cite[Prop.~2.3]{BBMN}, by
adapting \cite[Prop.~3.1]{DOW}.  We denote by $M\big(]0,T[\times \bb T
\times \bb R\big)$ the set of Radon measures on $]0,T[ \times \bb T
\times \bb R$. In the following, for $\varrho \in M\big(]0,T[\times
\bb T \times \bb R\big)$ we denote by $\varrho^\pm$ the positive and
negative part of $\varrho$.
\begin{proposition}
\label{p:kin}
Let $u \in \mc X$ be a weak solution to \eqref{e:1.1}. The following
statements are equivalent:
\begin{itemize}
\item[{\rm (i)}]{for each entropy $\eta$, the $\eta$-entropy
    production $\wp_{\eta,u}$ can be extended to a Radon measure on
    $]0,T[\times \bb T$;
}
\item[{\rm (ii)}]{there exists $\varrho_u(dv,dt,dx) \in M\big(\bb R
    \times ]0,T[\times \bb T\big)$, such that for any entropy $\eta$
    and $\varphi \in C^\infty_{\mathrm{c}}\big(]0,T[\times\bb T\big)$
\begin{equation}
\label{e:2.13}
\wp_{\eta,u}(\varphi) = 
\int \varrho_u(dv,dt,dx)\,\eta''(v)\varphi(t,x).
\end{equation}
}
\end{itemize}
\end{proposition}

A weak solution $u \in \mc X$ that satisfies the equivalent conditions
in Proposition~\ref{p:kin} is called an \emph{entropy-measure
  solution} to \eqref{e:1.1}. We denote by $\mc E_{u_0}$ the set of
entropy-measure solutions to \eqref{e:1.1} satisfying the initial
condition $u(0)=u_0$.
 
A weak solution $u\in \mc X$ to \eqref{e:1.1} is called an
\emph{entropic solution} iff for each convex entropy $\eta$ the
inequality $\wp_{\eta,u} \le 0$ holds. In particular entropic
solutions are entropy-measure solutions such that $\varrho_u$ is a
negative measure.  It is well known, see e.g.\ \cite{Pe}, that there
exists a unique entropic solution $\bar u\in C\left([0,T]; L_2(\bb
T)\right)$ to \eqref{e:1.1} such that $\bar u(0)=u_0$.

\subsection{Fractional hyperbolic entropy cost of non-entropic solutions} 
Recall that for $u \in \mc E_{u_0}$, $\varrho_u$ denotes its entropy
production measure as defined in Proposition~\ref{p:kin}, while
$\varrho^+$ is the positive part of $\varrho$. Define $I \colon \mc X
\to [0,+\infty]$ by
\begin{equation}
  \label{e:2.15}
  I(u):= 
  \begin{cases}
\displaystyle{
   \varrho_u^+(\bb R\times [0,T]\times \bb T)
             }
           & \text{if $u\in \mc E_{u_0}$}
 \\
+ \infty   & \text{ otherwise }
  \end{cases}
\end{equation}
namely $I(u)$ is the total variation of the positive part of the
entropy production of entropy-measure weak solutions to
\eqref{e:1.1}. The following proposition is proved in
\cite[Prop.~2.6]{BBMN}.
\begin{proposition}
\label{p:H}
The functional $I$ is lower semicontinuous on $\mc X$ and $I(u)=0$ iff
$u$ is an entropic solution to \eqref{e:1.1}.

Assume that there is no interval on $\bb R$ such that $f$ is affine
on such an interval. Then $I$ is coercive on $\mc X$.
\end{proposition}

The following theorem is the main result of this paper.
\begin{theorem}
 \label{t:ecne} 
\begin{itemize}
\item[{\rm (i)}] {The sequence of functionals $\{I_\eps\}$
    satisfies the $\Gamma$-liminf inequality
    $\Gamma$\textrm{-}$\varliminf_\eps I_\eps \ge I$ on $\mc X$.}

\item[{\rm (ii)}] {Assume that there is no interval on $\bb R$ such
    that $f$ is affine on such an interval. Then the sequence of
    functionals $\{I_\eps\}$ is equicoercive on $\mc X$.}
\end{itemize}
\end{theorem}
Note that Theorem~\ref{t:ecne} implies that the $\Gamma$-liminf
inequality holds even in weaker topologies. For instance, if $u_\eps
\to u$ in the sense of distributions, then still one has
$\varliminf_\eps I_\eps(u_\eps) \ge I(u)$. The $\Gamma$-liminf
inequality also implies some stability results for the fractional
viscous approximation to conservation laws, as shown in the next
corollary.
\begin{corollary}
\label{c:stab}
Let $E_\eps \in L_2([0,T]\times \bb T)$ be such that $\lim_\eps
\eps^{-1} \|E_\eps\|_{L_2([0,T]\times \bb T)}^2=0$. Then the solution
$u_\eps$ to \eqref{e:1.3} converges to the entropic solution of
\eqref{e:1.1}.
\end{corollary}
One may prove that the above corollary is sharp, in the sense that, if
$f$ is non-affine, for all $\delta>0$ there exists a sequence $E_\eps$
such that $\eps^{-1} \|E_\eps\|_{L_2([0,T]\times \bb T)} \le \delta$,
but $u_\eps$ converges to a solution to \eqref{e:1.1} which is not
entropic.

\subsection{Quasipotential}
\label{ss:quasipstat}
The functionals $I_{\eps}$ and $I$ as well as the space $\mc X$
introduced above depend on the time horizon $T$. In this section we
introduce in the notation the dependence on this parameter, so that
these objects will be denoted by $I_{\eps,T}$, $I_{T}$ and $\mc
X_{T}$.

Let $\int_{\bb T} u_{0}(x)\,dx=m \in \bb R$, then $I_{\eps,T}(u) =
I_{T}(u)=+\infty$ unless $\int_{\bb T} u(t,x)\,dx=m$ for each
$t\in[0,T]$. In addition, it is easy to see that the constant profile
$w(x)\equiv m$ is ''globally attractive'' for $I_{\eps,T}$, $I_T$, in
the sense that if $u \in C([0,+\infty[,L_2(\bb T))$ is such that
$I_{\eps,T}(u)$ or $I_T(u)$ are bounded uniformly in $T$, then $u$
will stay most of the time close to $m$. We are thus interested in
calculating the so-called quasipotential of the above functionals
starting at $m$.

More precisely, let $m\in \bb R$ define $V_{\eps}, V \colon \bb R
\times L_{2}(\bb T) \to [0,+\infty]$ as
\begin{equation}
\label{e:Veps1}
V_{\eps}(m;w):= \inf_{T>0} \inf_{\substack{u \in \mc X_{T} 
    \\ u(t=0)\equiv m,\,u(t=T)=w}} I_{\eps,T}(u)
\end{equation}
\begin{equation*}
  V(m;w):= \inf_{T>0} \inf_{\substack{u \in \mc X_{T} 
      \\ u(t=0)\equiv m,\,u(t=T)=w}} I_{T}(u)
\end{equation*}

Note that the definition of $V_\eps(m;w)$ and $V(m;w)$ also makes
sense out of $L_2(\bb T)$, but in view of \eqref{e:bound2} it is easily
seen that $V_\eps(m;w)=V(m;w)=+\infty$ if $w \not \in L_2(\bb T)$.

The following theorem gives an explicit characterization of $V_\eps$. 
\begin{theorem}
\label{t:quasipot}
It holds
\begin{equation*}
V_{\eps}(m;w)= 
\begin{cases}
\frac 12 \|w-m\|_{L_2(\bb T)}^2 & \text{if $\int_{\bb T}w(x)\,dx=m$}
\\
+\infty & \text{otherwise}
\end{cases}
\end{equation*}
\end{theorem}
$V$ has been calculated in \cite{BCM}, where it is shown that it
enjoys the same explicit representation as $V_\eps$ in
Theorem~\ref{t:quasipot} above.

\section{Local realization of the fractional Laplacian}
\label{s:3}
It is well known that one can see the operator
$(-\partial_{xx})^{1/2}$ on $\bb R$ by considering it as the Dirichlet
to Neumann operator associated to the harmonic extension in the
halfspace $\bb R \times \bb R^+$, paying the price to add a new
variable. In \cite{CS}, Caffarelli and Silvestre proved that this
is also possible for any power $s \in ]0,1[$ of
the Laplacian. In this section we shortly recall such a realization
when $\bb R$ is replaced by $\bb T$, together with a representation of
the bilinear form $(\varphi, \psi) \mapsto \langle (-\partial_{xx})^s
\varphi, \psi \rangle$, that will come useful later.

Hereafter in this paper, we denote by $\nabla$ the gradient operator
$\nabla= (\partial_x, \partial_y)$ on $\bb T \times \bb R^+$. Given $u
\in H^s(\bb T)$, we let $(x,y) \in \bb T \times \bb R^+ \mapsto \bar
u(x,y) \in \bb R$ be the unique solution to
\begin{equation}
\label{bdyFrac2} 
\begin{cases}
  \nabla \cdot (y^{1-2s} \nabla \bar u) = 0 
                  & \text{on $\bb T \times
    \bb R^+$}
  \\
  \bar u = u & \text{ on $\bb T \times \{y=0\}$}
\end{cases}
\end{equation} 
such that
\begin{equation}
  \|\bar u \|_{\dot H^1(\bb T\times \bb R^+, y^{1-2s})}^2 := 
          \int_{\bb T\times \bb R^+} y^{1-2s} |\nabla \bar u(x,y)|^2\,dx\,dy
\end{equation}
is finite. $\bar u$ is called the $s$-harmonic extension of $u$. The
following theorem is proved in \cite{CS}.

\begin{theorem}\label{t:realization}
  There exists a constant $c_s>0$ depending only on $s$, such that for
  every $u\in H^s(\bb T)$
\begin{equation*}
(-\partial_{xx})^s v= 
   - c_s \lim_{y \downarrow 0} y^{1-2s} \partial_y \bar u
\end{equation*}
where the equality holds in the distributional sense.
\end{theorem}  

In addition, it can be proved \cite{Nbi} that if $\tilde u \in \dot
H^{1}(\bb T \times \bb R^+,y^{1-2s})$, then $\tilde u$ can be traced
at $y=0$. By an integration by parts, it is then easy to verify that
for $u \in H^s(\bb T)$
\begin{equation} \label{argmin} 
\begin{split}
  & \| u \|_{\dot H^s(\bb T)}^2 
         = c_s\|\bar u \|_{\dot H^1(\bb T\times
                            \bb R^+, y^{1-2s})}^2 
\\ &
 \quad 
         = c_s \inf \big\{ \|\tilde u\|_{\dot H^1(\bb T\times \bb R^+, y^{1-2s})}^2
                          \,;\: \tilde u \in  \dot H^1(\bb T\times \bb R^+, y^{1-2s}),
                        \,\tilde u(x,0)=u(x) \big\}
\end{split}
\end{equation} 

The following remark is obtained by multiplying \eqref{bdyFrac2} by a
test function $ \tilde \varphi$, integrating the equation by parts on
$\bb T \times [\delta,+\infty[$, and passing to the limit $\delta\to
0$ thanks to Theorem~\ref{t:realization}.
\begin{remark}
\label{r:form}
Let $u \in H^s(\bb T)$ and $\varphi \in C^{\infty}(\bb T)$. Then
\begin{equation*}
\langle (-\partial_{xx})^s u, \varphi \rangle 
   = c_s \int_{\bb T\times \bb R^+} y^{1-2s} \nabla \bar u(x,y) 
      \, \cdot\, \nabla \tilde \varphi(x,y)\,dx\,dy
\end{equation*}
where $\tilde \varphi$ is any smooth, compactly supported function on
$\bb T\times \bb R^+$ such that $\tilde \varphi(x,0)=\varphi(x)$.
\end{remark}

\section{Regularity and a priori bounds}
\label{s:4}

\begin{proof}[Proof of Remark~\ref{r:riesz}]
  Assume that $I_\eps(u)<+\infty$. Then Riesz representation theorem
  for the dual spaces $\mc H$, $\mc H^*$ implies the identity
  \eqref{e:control} when the left and right hand sides are seen as
  elements of $\mc H^*$. Since $f$ is Lipschitz and bounded and $s >
  1/2$, $\partial_x f(u) \in L_2([0,T]; H^{s-1}(\bb T)) \subset \mc
  H$; moreover $(-\partial_{xx})^s u \in \mc H^*$, thus all of the
  right hand side terms of the equation are separately in $\mc H^*$ as
  well, and the equation holds weakly when each single term is checked
  against test functions in $\mc H$. In particular, it holds weakly
  against smooth functions.
\end{proof}

\begin{lemma}
  \label{l:ext1} 
There exists a constant $C_\eps>0$ such that for all $u$
  with $I_\eps(u)<+\infty$
\begin{equation}
\label{e:bound1}
\eps \|u\|_{\mc H}^2\le 2 \|u_0\|_{L_2(\bb T)}^2 + 4 I_\eps(u)
\end{equation}
\begin{equation}
\label{e:bound2}
\sup_{t\in [0,T]} \|u(t)\|_{L_2(\bb T)}^2 
   \le 2 \|u_0\|_{L_2(\bb T)}^2 + 4 I_\eps(u)
\end{equation}
\begin{equation}
\label{e:bound3}
\|\partial_t u\|_{L_2([0,T];H^{-s}(\bb T))}^2 
   \le C_\eps \big[\|u_0\|_{L_2(\bb T)}^2 + I_\eps(u) \big]
\end{equation}
\begin{equation}
\label{e:bound4}
\|u\|_{H^{1/2}([0,T];L_2(\bb T))}^2 
   \le C_\eps \big[\|u_0\|_{L_2(\bb T)}^2 + I_\eps(u) \big]
\end{equation}
\end{lemma}

\begin{proof}
By Remark~\ref{r:riesz}, \eqref{e:control} reads
\begin{equation}
\label{e:ineq1}
\begin{split}
& \langle u(t), \varphi(t) \rangle  
    - \langle u_0, \varphi(0) \rangle     
    - \int_0^t \big[\langle u(r), \partial_t \varphi(r) \rangle 
    + \langle f(u(r)), \partial_x \varphi(r) \rangle \big] \,dr
\\    & \qquad 
+ \frac{\eps}2 \int_0^t \langle (-\partial_{xx})^{s/2} u(r), 
                         (-\partial_{xx})^{s/2} \varphi(r) \rangle\,dr
\\ &
=  \int_0^t \langle -(\partial_{xx})^{s/2} \Phi(r), 
                     (-\partial_{xx})^{s/2} \varphi(r) \rangle\,dr
\end{split}
\end{equation}
for any $t\in [0,T]$, $\varphi \in C^\infty([0,T]\times \bb T)$. Note
that since $f$ is Lipschitz and bounded, $f(u) \in L_2([0,T];H^s(\bb
T))$. Therefore, by a density argument it is easy to see that
\eqref{e:ineq1} holds for any $\varphi \in L_2([0,T];H^s(\bb T))$ such
that $\partial_t \varphi,\,\partial_x \varphi \in \mc H^*$. Recalling
that $\partial_t u \in \mc H^*$, that $u\in L_2([0,T];H^s(\bb T))$,
and thus, as $s>1/2$, $\partial_x u \in L_2([0,T],\dot H^{s-1}(\bb T))
\subset \mc H^*$, the choice $\varphi = u$ is allowed in
\eqref{e:ineq1}. In view of $u \in C([0,T];L_2(\bb T))$, with the same
duality argument it is now immediate to verify that integration by
parts are allowed in \eqref{e:ineq1} with $\varphi=u$. Since $\langle
f(u),\partial_x u \rangle =0$, by Cauchy-Schwarz inequality and
\eqref{e:Irep}
 \begin{equation*}
 \begin{split}
& \frac{1}{2} \|u(t)\|_{L_2(\bb T)}^2 
         -  \frac{1}{2} \|u_0\|_{L_2(\bb T)}^2 + \frac{\eps}2 \|u\|_{\mc H}^2
\\ &
= \int_0^t \langle -(\partial_{xx})^{s/2} \Phi(r), 
                   (-\partial_{xx})^{s/2} u(r) \rangle\,dr \le \|u\|_{\mc H}\,\|\Phi\|_{\mc H}
\\ &
\le \frac{\eps}4 \|u\|_{\mc H}^2 + 2I_\eps(u)
\end{split}
 \end{equation*}
 Passing to the supremum in $t$ one gets \eqref{e:bound1},
 \eqref{e:bound2}. \eqref{e:bound3} is then obtained by
 \eqref{e:control} and \eqref{e:bound1},
 \eqref{e:bound2}. \eqref{e:bound4} follows from the fact that
 $H^{1/2}([0,T];L_2(\bb T))$ is the Hilbert interpolation of parameter
 $1/2$ between $H^1([0,T];H^{-s}(\bb T))$ and $L_2([0,T];H^s(\bb T))$,
 while \eqref{e:bound1}-\eqref{e:bound3} grant the bounds in these
 latter spaces.
\end{proof}

\begin{proof}[Proof of Proposition~\ref{p:semicont}: lower semicontinuity]
  Let $v^n$ be a sequence converging to $v$ in $\mc X$, such that
  $I_\eps(v^n)$ is bounded uniformly in $n$. We need to show
  $\varliminf_n I_\eps(v^n)\ge I_\eps(v)$. By the uniform bound on
  $I_\eps(v^n)$, \eqref{e:bound1} and the lower semicontinuity of the
  Hilbert norm, we have that $v\in L_2([0,T]; H^s(\bb T))$. Still by
  Lemma~\ref{l:ext1}, \eqref{e:bound2}-\eqref{e:bound3} and the
  embedding of $H^{1/2}([0,T];L_2(\bb T))$ in $C([0,T];L_2(\bb T))$,
  we obtain that $v\in C([0,T];L_2(\bb T))$.

  Note that if $I_\eps(u)<+\infty$, for $\Phi$ as in
  Remark~\ref{r:riesz}
\begin{equation}
\label{e:Irap}
\begin{split}
I_\eps(u)= 
& \eps^{-1} \sup_{\phi \in \mc H} 
           \langle -\partial_{xx} \Phi,\phi \rangle_{\mc H^*,\mc H}
                 -\frac 12 \|\phi\|_{\mc H}^2 
\\ = & \eps^{-1} 
\sup_{\varphi \in C^{\infty}([0,T]\times \bb T)}
\langle u(T), \varphi(T) \rangle  
    - \langle u_0, \varphi(0) \rangle     
\\ &
   - \int_0^T \big[\langle u(r), \partial_t \varphi(r) \rangle 
 + \langle f(u(r)), \partial_x \varphi(r) \rangle \big] \,dr
\\    & 
+ \frac{\eps}2 \int_0^T \langle  u(r), 
         (-\partial_{xx})^{s} \varphi(r) \rangle\,dr
-\frac 12 \int_0^T \big\|(-\partial_{xx})^{s/2}
           \varphi(r) \big\|_{L_2(\bb T)}^2\,dr
\end{split}
\end{equation}
while the latter sup in the above formula equals $+\infty$ if $u\in
C([0,T];L_2(\bb T)) \cap \mc H $ but $\partial_t u
+\partial_x f +\frac{\eps}2 (-\partial_{xx})^s u \not \in \mc
H^*$. Thus \eqref{e:Irap} represents the restriction of $I_\eps$ to $
C([0,T];L_2(\bb T)) \cap \mc H$ as a supremum of continuous functions
on $\mc X$. Since, as remarked at the beginning of the proof, one can indeed restrict to the case $v_n,\,v \in C([0,T];L_2(\bb T)) \cap \mc H$, the lower semicontinuity follows.
\end{proof}

\begin{lemma}
\label{l:conth}
If $I_\eps(u)<+\infty$, there exists a constant $C>0$ depending only
on $f$, $T$, and $s$ such that for all $r,\,t \in [0,T]$ and $\varphi
\in H^1(\bb T)$
\begin{equation*}
\big| \langle u(t)-u(r),\varphi \rangle\big| 
       \le C (1+I_\eps(u)) |t-r|^{1/2} \|\varphi\|_{\dot H^1(\bb T)}
\end{equation*}
\end{lemma}
\begin{proof}
  For $\varphi \in C^{\infty}(\bb T)$, by \eqref{e:control}, the
  Cauchy-Schwarz inequality and Sobolev embedding
\begin{equation*}
\begin{split}
 \big| \langle u(t) -u(r),\varphi\rangle \big|  = & 
\Big| \int_r^t \langle f(u(r')),\partial_x \varphi \rangle\,dr'
 +\frac{\eps}2 \int_r^t \langle (-\partial_{xx})^{s/2} u(r') ,
                            (-\partial_{xx})^{s/2} \varphi \rangle\,dr'
\\ & 
  -\int_r^t \langle (-\partial_{xx})^{s/2} \Phi(r') ,
                (-\partial_{xx})^{s/2} \varphi \rangle \Big|\,dr'
\\  
\le & 
|t-s| \big[ \sup_{w\in \bb R}|f(w)| \big] \|\partial_x \varphi\|_{L_2(\bb T)}
 +  |t-s|^{1/2}\frac{\eps}{2} \|u\|_{\mc H} \|\partial_x \varphi\|_{L_2(\bb T)}
 \\ &
 + |t-s|^{1/2} \|\Phi \|_{\mc H} \|\partial_x \varphi\|_{L_2(\bb T)} 
\end{split}
\end{equation*}
The proof is concluded using \eqref{e:Irep} and \eqref{e:bound1}.
\end{proof}

\begin{proof}[Proof of Proposition~\ref{p:semicont}: coercivity]
  We want to prove that if a sequence $(v^n)$ is such that
  $I_\eps(v^n)$ is uniformly bounded in $n$, then $(v^n)$ is
  precompact in $\mc X$. By Lemma~\ref{l:conth}, $(v^n)$ is precompact
  in $C([0,T];H^{-1}(\bb T))$. Let $v$ be any limit point of
  $(v^n)$. Up to passing to a subsequence, we can assume $v^n\to v$ in
  $C([0,T];H^{-1}(\bb T))$. By \eqref{e:bound2} $v\in L_2([0,T]\times
  \bb T)$, and  it is enough to prove that $v^n$ converges to $v$
  strongly in $L_2$ to conclude.

  By \eqref{e:bound1}, $v^n$ stays bounded in $\mc H$ and thus it
  converges to $v$ weakly in $\mc H$. Let $\jmath$ be a smooth
  convolution kernel on $\bb T$, and let $\ast$ denote convolution in
  space. Then
\begin{equation}
\label{e:treps}
\begin{split}
&\big\| v^{n} - v \big\|_{L_2([0,T]\times \bb T)} \le 
\big\|\jmath * v^{n} -  \jmath * v
                 \big\|_{L_2([0,T]\times \bb T)}
\\ &
\qquad \qquad 
 +\big\|\jmath * v - v
                 \big\|_{L_2([0,T]\times \bb T)}
      +       \big\| v^n - \jmath * v^n 
                 \big\|_{L_2([0,T]\times \bb T)}
\end{split}
\end{equation}
The first term in the right hand side vanishes as $n\to +\infty$ by
the convergence of $v^n$ in $C([0,T];H^{-1}(\bb T))$. The second term
vanishes if we let $\jmath$ converge to the Dirac mass at $0$. As for
the third term, by Sobolev embedding, there exists $\imath \in \dot
H^{-a}(\bb T)$ for $a>1/2$, such that
\begin{equation*}
\partial_x \imath=  \delta_0-\jmath
\end{equation*}
in the distribution sense.Then
\begin{equation*}
\begin{split}
\big\| v^n - \jmath * v^{n} \big\|_{L_2([0,T]\times\bb T)} 
& =
\Big\| (-\partial_{xx})^{1/2} \imath * v^n  \Big\|_{L_2([0,T]\times\bb T)}
\\ &
= \Big\| [(-\partial_{xx})^{\frac{1-s}{2}}\imath] * 
           [(-\partial_{xx})^{s/2} v^n]  \Big\|_{L_2([0,T]\times\bb T)} 
\\
& \le \sqrt{T} \Big\| (-\partial_{xx})^{\frac{1-s}{2}}\imath \Big\|_{L_1(\bb T)}
   \: \big\|  (-\partial_{xx})^{s/2} v^{n} \big\|_{L_2([0,T]\times\bb R)}
   \\
&
   \le \sqrt{T}\, \|\imath\|_{\dot H^{1-s}} \|v^n\|_{\mc H}
\end{split}
\end{equation*}
where in the third line we used Young inequality. By \eqref{e:bound1},
$\|v^n\|_{\mc H}$ is bounded uniformly in $n$, while $\|\imath\|_{\dot
  H^{1-s}}$ vanishes as we let $\jmath$ converge to $\delta_0$, since
$1-s<1/2$. Therefore all of the terms in the right hand side of
\eqref{e:treps} vanish, as we let $n\to +\infty$ first, and $\jmath
\to \delta_0$ next.
\end{proof}

\section{Equicoercivity and the $\Gamma$-liminf inequality}
\label{s:5}
In this section we prove Theorem~\ref{t:ecne} and
Corollary~\ref{c:stab}.

\begin{proof}[Proof of Theorem~\ref{t:ecne}-(i)]
  Let $u_\eps$ be a sequence converging to $u$ in $\mc X$. We want to
  prove that $\varliminf_{\eps} I_\eps(u_\eps) \ge I(u)$. With no loss
  of generality, we assume that $I_\eps(u_\eps)$ is uniformly
  bounded. Thus, by the convergence in $C([0,T];H^{-1}(\bb T))$, one
  has $u(t=0)=\lim_\eps u_\eps(0)=u_0$.
  
  Let $\vartheta \in C^2_{\mathrm{c}}(\bb R \times ]0,T[ \times \bb
  T)$. Let $Q \in C^1_{\mathrm{c}}(\bb R \times ]0,T[\times \bb T)$ be
  defined (up to an additive function of $(t,x)$) by
\begin{equation*}
Q'(v,t,x)= \int^v f'(w) \,\vartheta'(w,t,x)\,dw
\end{equation*}
where $\vartheta',\,Q'$ denote derivatives with respect to the first
variable. We will also denote $\vartheta_t$, $Q_t$ and $\vartheta_x$,
$Q_x$ the partial derivatives with respect to the second and third
arguments of $\vartheta$ and $Q$ respectively.  By \eqref{e:Irap}, for
all smooth $\varphi \in C^{\infty}_{\mathrm{c}}(]0,T[\times \bb T)$
and $\tilde \varphi \in C^{\infty}_{\mathrm{c}}(]0,T[\times \bb T
\times \bb R)$ such that $\tilde \varphi(t,x,0)=\varphi(t,x)$
\begin{equation*}
\begin{split}
I_\eps(u_\eps) \ge
 &
   - \eps^{-1} \int_0^T \big[\langle u_\eps(r), \partial_t \varphi(r) \rangle 
 + \langle f(u_\eps(r)), \partial_x \varphi(r) \rangle \big] \,dr
\\    & 
+ \frac{1}2 \int_0^T \langle (-\partial_{xx})^{s/2} u_\eps(r), 
                             (-\partial_{xx})^{s/2} \varphi(r) \rangle\,dr
\\ &
-\frac{\eps^{-1}}2 \int_0^T \big\|(-\partial_{xx})^{s/2} \varphi(r) \big\|_{L_2(\bb T)}^2\,dr
\\  
\ge &
 - \eps^{-1}\int_0^T \big[\langle u_\eps(r), \partial_t \varphi(r) \rangle 
 + \langle f(u_\eps(r)), \partial_x \varphi(r) \rangle \big] \,dr
\\    & 
+ \frac{1}2 \int_0^T \int y^{1-2s} \nabla \bar u_\eps(r,x,y)\, \cdot\, 
                                  \nabla \tilde \varphi(r,x,y) \,dx\,dy\,dr
\\ &
- \frac{\eps^{-1}}2 \int_0^T \int y^{1-2s} \nabla \tilde \varphi(r,x,y)\, 
                                 \cdot\, \nabla \tilde \varphi(r,x,y) \,dx\,dy\,dr
\end{split}
\end{equation*}
where the last inequality follows from \eqref{argmin}.

Let $\chi \in C^{\infty}_{\mathrm{c}}(\bb R^+;[0,1])$, such that
$\chi(0)=1$. By the same density argument used in the proof of
Lemma~\ref{l:ext1}, one can indeed plug $\varphi(t,x)=\eps
\vartheta'(u_\eps(t,x),t,x)$ as a test function above. The key point
is now the choice $\tilde \varphi(t,x,y)=\eps \vartheta'(\bar
u_\eps(t,x,y),t,x)\chi(y)$, which is indeed an extension of $\varphi$,
though not the $s$-harmonic one. Again reasoning as in
Lemma~\ref{l:ext1}, integrations by parts are allowed, so that
\begin{equation*}
\begin{split}
& I_\eps(u_\eps) \ge
    - \int_0^T\int_{\bb T} \vartheta_t(u_\eps(r,z),r,z) 
                       + Q_x(u_\eps(r,z),r,z) \,dz\,dr
\\    & 
+ \frac{\eps}2 \int_0^T \int y^{1-2s}\chi(y) \vartheta''(\bar u_\eps(r,z,y),r,z) 
                    \nabla \bar u_\eps(r,x,y)\, \cdot\, \nabla \bar u_\eps(r,x,y)\,dx\,dy\,dr
\\ &
+ \frac{\eps}2 \int_0^T \int y^{1-2s}
\begin{pmatrix}
\vartheta'_x(\bar u_\eps(r,z,y),r,z) \chi(y) \\
\vartheta'(\bar u_\eps(r,z,y),r,z) \chi'(y) 
\end{pmatrix}
\cdot \,\nabla \bar u_\eps(r,x,y)\,dx\,dy\,dr
\\ &
- \frac{\eps}2\int_0^T \int y^{1-2s}\chi(y)^2 
          \vartheta''(\bar u_\eps(r,z,y),r,z)^2
          \nabla \bar u_\eps(r,x,y)\, \cdot\, \nabla \bar u_\eps(r,x,y)\,dx\,dy\,dr
\\ &
- \frac{\eps}2 \int_0^T \int y^{1-2s} \vartheta''(\bar u_\eps(r,z,y),r,z) \chi(y))
\begin{pmatrix}
\vartheta'_x(\bar u_\eps(r,z,y),r,z) \chi(y) \\
\vartheta'(\bar u_\eps(r,z,y),r,z) \chi'(y) 
\end{pmatrix}
\cdot \,\nabla \bar u_\eps(r,x,y)\,dx\,dy\,dr
\\ &
- \eps \int_0^T \int y^{1-2s} 
\begin{pmatrix}
\vartheta'_x(\bar u_\eps(r,z,y),r,z) \chi(y) \\
\vartheta'(\bar u_\eps(r,z,y),r,z) \chi'(y) 
\end{pmatrix}
\cdot \,\begin{pmatrix}
\vartheta'_x(\bar u_\eps(r,z,y),r,z) \chi(y) \\
\vartheta'(\bar u_\eps(r,z,y),r,z) \chi'(y) 
\end{pmatrix}\,dx\,dy\,dr
\end{split}
\end{equation*}
The main idea now is that the third, fifth and sixth lines vanish as
$\eps\to 0$, while the second and forth line compensate (each
being order $1$) for a suitable class of $\vartheta$. Indeed, since
$\vartheta$ and $\chi$ have bounded derivatives up to the second
order, and $y^{1-2s}$ is integrable at $0$, the Cauchy-Schwarz
inequality yields for the third, fifth and sixth lines
\begin{equation*}
\begin{split}
& I_\eps(u_\eps) \ge
    - \int_0^T\int_{\bb T} \vartheta_t(u_\eps(r,z),r,z) + Q_x(u_\eps(r,z),r,z) \,dz\,dr
\\    & 
+ \frac{\eps}2 \int_0^T \int y^{1-2s} \big[\vartheta''(\bar u_\eps(r,z,y),r,z)\chi(y)-
\vartheta''(\bar u_\eps(r,z,y),r,z)^2\chi(y)^2\big]
\\ & \qquad \qquad \qquad 
 \nabla \bar u_\eps(r,x,y)\, \cdot\, \nabla \bar u_\eps(r,x,y)\,dx\,dy\,dr
\\ &
+ \eps\,C\,\big[1+\|u_\eps\|_{\mc H}\big]
\end{split}
\end{equation*}
for a constant $C$ depending only on $\vartheta$, $\chi$, $T$ and $s$.
By the bound \eqref{e:bound1}, recalling that $I_\eps(u_\eps)$ is
uniformly bounded, the last term vanishes as $\eps \to 0$. Note
moreover that if $0 \le \vartheta'' \le 1$ the second line in the last
formula is positive. Recalling that $u_\eps \to u$ strongly in
$L_2([0,T]\times \bb T)$, passing to the limit $\eps \to 0$ and
optimizing over smooth $\vartheta$ for which the inequality holds one
thus obtains
\begin{equation*}
\varliminf_{\eps \downarrow 0} I_\eps(u_\eps) \ge
 \sup_{\vartheta\,:\: 0\le \vartheta''\le 1}   - \int_0^T\int_{\bb T} 
      \vartheta_t(u(r,z),r,z) + Q_x(u(r,z),r,z) \,dz\,dr
\end{equation*}
In \cite[Formula (5.1)--(5.2)]{BBMN} it is proved that the supremum in
the right hand side of the above formula equals $I(u)$, provided
$u(0)=u_0$ (which we already know).
\end{proof}

\begin{lemma}
\label{l:compentr}
Let $(u_\eps)$ be a sequence in $\mc X$ such that $I_\eps(u_\eps)$ is
bounded uniformly in $\eps$.  Let $(\eta,q)$ be an entropy-entropy
flux pair, with $\eta$ bounded with bounded first and second
derivatives. Recall that $\wp_{\eta,u_\eps}$ is the distribution
$\wp_{\eta,u_\eps}:=\partial_t \eta(u_\eps) + \partial_x
q(u_\eps)$. Then $(\wp_{\eta,u_\eps})$ is strongly compact in
$H^{-1}([0,T]\times \bb T)$.
\end{lemma}
\begin{proof}
  Let $\varphi \in C^{\infty}([0,T]\times \bb T)$. Reasoning as in
  Lemma~\ref{l:ext1}, one is allowed to test equation
  \eqref{e:control} against the test function $(t,x) \mapsto
  \eta'(u_\eps(t,x)) \varphi(t,x)$. Still by the same argument as in
  Lemma~\ref{l:ext1}, integrations by parts are allowed so that (in
  the following, we denote $\Phi_\eps \equiv \Phi_{u_\eps}$).
\begin{equation*}
\begin{split}
& \wp_{\eta,u_\eps}(\varphi)  =  \langle \eta(u_\eps(T)),\varphi(T) \rangle 
                            - \langle \eta(u_\eps(0)),\varphi(0) \rangle
\\ & \qquad \quad
- \int_0^T \big[ \langle \eta(u_\eps(t)),\partial_t \varphi(t) \rangle + 
\langle q(u_\eps(t)),\partial_x \varphi(t) \rangle \big]\,dt
\\ & \qquad
=-\frac{\eps}2 \int_0^T \langle \eta'(u_\eps(t)) \varphi(t), 
                               (-\partial_{xx})^s u_\eps(t) \rangle\,dt
\\ &
\qquad \quad + \int_0^T \langle \eta'(u_\eps(t)) \varphi(t), 
                          (-\partial_{xx})^s \Phi_\eps(t) \rangle\,dt
\end{split}
\end{equation*}
Recall that for $v \in H^s(\bb T)$, we denoted $\bar v$ its
$s$-harmonic extension of $v$ to $\bb T\times \bb R^+$. Note that
$\eta'(\bar u_\eps) \bar \varphi$ provides an extension (thought not
$s$-harmonic) of $\eta'(u_\eps) \varphi$. Therefore, by
Remark~\ref{r:form} applied with $\tilde \varphi = \eta'(\bar u_\eps)
\bar \varphi$, we get
\begin{equation*}
\begin{split}
& \wp_{\eta,u_\eps}(\varphi)  
=-\frac{\eps}2 \int_0^T \int_{\bb T \times \bb R^+} y^{1-2s} \eta''(\bar u_\eps(t,x,y))
                  \,\bar \varphi(t,x,y) \nabla u_\eps(t,x,y) 
                  \,\cdot \, \nabla u_\eps(t,x,y)\,dx\,dy\,dt
\\ & \qquad
-\frac{\eps}2 \int_0^T \int_{\bb T \times \bb R^+} y^{1-2s} \eta'(\bar u_\eps(t,x,y))
                  \,\nabla \bar \varphi(t,x,y) \,\cdot\, \nabla u_\eps(t,x,y)\,dx\,dy\,dt
\\ & \qquad
+ \int_0^T \int_{\bb T \times \bb R^+} y^{1-2s} \eta''(\bar u_\eps(t,x,y))\,
                 \bar \varphi(t,x,y) \nabla u_\eps(t,x,y) \,\cdot 
              \, \nabla \Phi_\eps(t,x,y)\,dx\,dy\,dt
\\ & \qquad
+\int_0^T \int_{\bb T \times \bb R^+} y^{1-2s} \eta'(\bar u_\eps(t,x,y))
        \,\nabla \bar \varphi(t,x,y) \,\cdot
        \, \nabla \Phi_\eps(t,x,y)\,dx\,dy\,dt
\end{split}
\end{equation*}
Since $\eta$ has bounded derivatives, recalling the variational
definition of the $H^s$-norm \eqref{argmin}, the Cauchy-Schwarz
inequality yields for a constant $C$ depending only on $\eta$ and $T$
\begin{equation*}
\begin{split}
 \big|\wp_{\eta,u_\eps}(\varphi) \big|
\le & \eps\,C \|\bar \varphi\|_{\infty}\,\|u_\eps\|_{\mc H}^2
+ \eps\,C\,\|u_\eps\|_{\mc H} \|\varphi\|_{\mc H}
\\ &
+C\,\|\bar \varphi\|_{\infty} \|u_\eps\|_{\mc H} \|\Phi\|_{\mc H}
+C\, \|\varphi\|_{\mc H} \|\Phi\|_{\mc H}
\end{split}
\end{equation*}
The maximum principle holds for the $s$-harmonic extension,
$\|\bar \varphi\|_{\infty} =\|\varphi\|_{\infty}$, so that using
\eqref{e:bound1}, \eqref{e:Irep}
\begin{equation*}
\begin{split}
 \big|\wp_{\eta,u_\eps}(\varphi) \big| \le C' \big(1+I_\eps(u_\eps)\big) 
           \big( \|\varphi\|_{\infty} + \sqrt{\eps} \|\varphi\|_{\mc H} \big)
\end{split}
\end{equation*}
for a suitable constant $C'$ independent of $\eps$. The last
inequality implies that $\wp_{\eta,u_\eps}$ can be written as the sum
of two distributions $\wp_{\eta,u_\eps}=\wp^1_\eps+\wp^2_\eps$, where
$\wp^1_\eps$ is a finite measure on $[0,T]\times \bb T$ with total
variation bounded uniformly in $\eps$, while $\wp^2_\eps$ has
vanishing $\mc H^*$-norm. By Sobolev compact embedding, both
$\wp^1_\eps$ and $\wp^2_\eps$ are compact in $H^{-1}([0,T]\times \bb
T)$, and thus $\wp_{\eta,u_\eps}$ is.
\end{proof}

Before proving Theorem~\ref{t:ecne}-(ii) we recall some standard facts
concerning Young measures. Let $(u^\eps)$ be a sequence in $\mc X$
such that $\|u_\eps\|_{L_2([0,T]\times \bb T)}$ is uniformly
bounded. Then the sequence of Radon measures
$\delta_{u_\eps(t,x)}(d\lambda)dt\,dx$ over $\bb R\times [0,T] \times
\bb T$ is compact in the weak* topology of Radon measures, and any
limit point can be represented by a Young measure, namely a measurable
map $[0,T]\times \bb T \ni (t,x) \mapsto \mu_{t,x} \in \mc P(\bb R)$
such that, up to passing to subsequences
\begin{equation}
\label{e:youngconv}
\lim_{\eps \downarrow 0} 
     \int_{[0,T]\times \bb T} F(u_\eps(t,x)) \varphi(t,x)\,dt\,dx 
   = \int_{\bb R \times [0,T]\times \bb T} F(\lambda)\varphi(t,x) 
        \,\mu_{t,x}(d\lambda)\, \,dt\,dx
\end{equation}
for all continuous compactly supported test functions $F$ and
$\varphi$. Moreover, $u_\eps$ converges strongly in $L_r([0,T]\times
\bb T)$ for $r<2$ to a $u$ iff $\mu_{t,x}=\delta_{u(t,x)}$.

\begin{proof}[Proof of Theorem~\ref{t:ecne}-(ii)]
  Let $(u_\eps)$ be a sequence in $\mc X$ such that $I_\eps(u_\eps)$
  is bounded uniformly in $\eps$. Lemma~\ref{l:conth} and
  Ascoli-Arzela theorem imply that $u_\eps$ is compact in
  $C([0,T];H^{-1}(\bb T))$. Therefore we need to prove that $u_\eps$
  is strongly compact in $L_p([0,T]\times \bb T)$ to conclude the
  proof.

  In view of the uniform (in $\eps$) bound \eqref{e:bound2} for
  $u_\eps$, there exists a Young measure $ \mu$ such that
  \eqref{e:youngconv} holds (up to passing to a suitable subsequence
  still labeled by $\eps$). We need to prove that there exists $u \in
  L_2([0,T]\times \bb T)$ such that $\mu_{t,x}=\delta_{u(t,x)}$ for
  a.e.\ $(t,x)$. To achieve this, we will follow a celebrated argument
  by Tartar, \cite[Chap.~9]{Se}. However, since we here lack the usual
  $L_\infty$ bounds used in this approach, we shortly reproduce the
  argument adapted to our setting.

  Following closely \cite[Chap.~9]{Se}, thanks to
  Lemma~\ref{l:compentr} one has for a.e.\ $(t,x)$
\begin{equation}
\label{e:iden}
\int_{\bb R^2} [\eta_1(\xi)-\eta_1(\zeta)] [q_2(\xi)-q_2(\zeta)] 
           - [\eta_2(\xi)-\eta_2(\zeta)] [q_1(\xi)-q_1(\zeta)] 
             \mu_{(t,x)}(d\xi)\mu_{(t,x)}(d\lambda) =0
\end{equation}
for $\eta_1$, $\eta_2$ two smooth bounded entropies with bounded
derivatives, and $q_1$, $q_2$ their respective conjugated entropy
fluxes. By a density argument, \eqref{e:iden} is easily seen to hold
for $\eta_1,\,\eta_2$ Lipschitz and uniformly bounded. Fix $M>0$ and
take
\begin{equation*}
\eta_1(v)=
\begin{cases}
v & \text{if $v \in [-M,M]$}
\\
-M & \text{if $v <-M$}
\\
M  & \text{if $v >M$}
\end{cases}
\end{equation*}
\begin{equation*}
\eta_2(v)=
\begin{cases}
f(v) & \text{if $v \in [-M,M]$}
\\
f(-M) & \text{if $v <-M$}
\\
f(M)  & \text{if $v >M$}
\end{cases}
\end{equation*}
Thus \eqref{e:iden} now reads
\begin{equation*}
\begin{split}
\int_{\bb R^2} & \Big[
\Big(\int_{\xi}^{\zeta} f'(a) \un {[-M,M]}(a)\,da
\Big)^2
\\ & \:
 - \Big(\int_{\xi}^{\zeta} \un {[-M,M]}(a)\,da
\Big) \Big(\int_{\xi}^{\zeta} f'(a)^2 \un {[-M,M]}(a)\,da
\Big) \Big]
\mu_{(t,x)}(d\xi)\mu_{(t,x)}(d\lambda) =0
\end{split}
\end{equation*}
By Cauchy-Schwarz inequality, the integrand in square brackets is
always negative, vanishing iff $ f'(a) \un {[-M,M]}(a)$ is constant
between $\xi$ and $\zeta$. Since we assumed that there is no interval
in which $f$ is affine, this implies that the restriction of
$\mu_{(t,x)}$ to $[-M,M]$ is a Dirac mass for a.e.\ $(t,x)$. Since $M$
is arbitrary, we conclude.
\end{proof}

\begin{proof}[Proof of Corollary~\ref{c:stab}]
If we let $\Phi_\eps$ be the solution to
\begin{equation*}
(-\partial_{xx})^s \Phi_\eps = (-\partial_{xx})^{s/2} E_\eps
\end{equation*}
then
\begin{equation*}
I_\eps(u_\eps)= \frac{\eps^{-1}}2 \|\Phi_\eps\|_{\mc H}^2 
           \le  \frac{\eps^{-1}}{2} \|E_\eps\|_{L_2([0,T]\times \bb T)}^2
\end{equation*}
Therefore $I_\eps(u_\eps) \to 0$, and by Theorem~\ref{t:ecne}-(ii), up
to passing to subsequences, $u_\eps \to u$ in $\mc X$ for a suitable
$u\in \mc X$. By Theorem~\ref{t:ecne}-(i), $I(u) \le \varliminf_\eps
I_\eps(u_\eps) = 0$. Thus $I(u)=0$ and by Proposition~\ref{p:H}, $u$
is the (unique) entropic solution to \eqref{e:1.1}.
\end{proof}

\section{Quasipotential}
\label{s:6}
In this section we prove Theorem~\ref{t:quasipot}. As in
Section~\ref{ss:quasipstat}, here we append a subscript $T$ to the
notation, to stress the dependence on $T$.

If $\int_{\bb T} w(x)\,dx \neq m$, then Theorem~\ref{t:quasipot}
follows from the conservation of the total mass of $L_2$-solutions to
\eqref{e:control}. Namely, if $u \in \mc X_T$ is such that
$I_{\eps;T}(u)<+\infty$, then $\int_{\bb T} u(t,x)\,dx = \int_{\bb T}
u(0,x)\,dx$ for all $t \ge 0$, and thus the infimum in \eqref{e:Veps1}
equals $+\infty$.

So hereafter in this section we assume $\int_{\bb T}w(x)\,dx= m$. Then
the proof of Theorem~\ref{t:quasipot} is a consequence of the
following Lemmata. In fact from Lemma~\ref{l:ubve} one gets
$V_\eps(m,w)\ge \frac 12 \|w-m\|_{L_2(\bb T)}^2$, and from
Lemma~\ref{l:ubve} and Lemma~\ref{l:lbve} one has $V_\eps(m,w)\le
\frac 12 \|w-m\|_{L_2(\bb T)}^2 + \gamma$ for each $\gamma>0$.
\begin{lemma}
\label{l:ubve}
Let $T>0$ and $u\in \mc X_T$ be such that
$I_{\eps;T}(u)<+\infty$, $u(0,x)\equiv m$, $u(T,x)=w(x)$. Then
\begin{equation*}
I_{\eps;T}(u) = \frac 12 \|w-m\|_{L_2(\bb T)}^2
+\frac{\eps^{-1}}{2} \Big\|\partial_t u + \partial_x f(u)
         - \frac{\eps}2(-\partial_{xx})^s u
          \Big\|_{\mc H_T^*}^2
\end{equation*}
\end{lemma}

\begin{lemma}
\label{l:lbve}
For each $\gamma>0$, there exists $T>0$ and $u\in \mc X_T$ such that
$I_{\eps;T}(u)<+\infty$, $u(0)\equiv m$, $u(T) = w$ and
\begin{equation}
\label{e:resto}
\frac{\eps^{-1}}{2} \Big\|\partial_t u + \partial_x f(u)
         - \frac{\eps}2(-\partial_{xx})^s u
          \Big\|_{\mc H_T^*}^2 \le \gamma
\end{equation}
\end{lemma}

\begin{proof}[Proof of Lemma~\ref{l:ubve}]
  Since $I_{\eps;T}(u)<+\infty$, as observed in the proof of
  Remark~\ref{r:riesz}, $u \in \mc H_T$, $\partial_t u \in \mc H^*_T$,
  $\Phi_u \in \mc H_T$, and $\partial_x f(u) \in L_2([0,T];\dot
  H^{s-1}(\bb T)) \subset \mc H^*$ as $s>1/2$. Therefore, by
  decomposition of Hilbert scalar products
\begin{equation}
\label{e:dec}
\begin{split}
I_{\eps;T}(u)   =  &
\frac{\eps^{-1}}{2} \Big\|\partial_t u + \partial_x f(u)
         + \frac{\eps}2(-\partial_{xx})^s u
          \Big\|_{\mc H_T^*}^2
\\
  =&
\frac{\eps^{-1}}{2} \Big\|\partial_t u + \partial_x f(u)
         - \frac{\eps}2(-\partial_{xx})^s u
          \Big\|_{\mc H_T^*}^2
\\  &
 + \Big( \partial_t u, 
         (-\partial_{xx})^s u
          \Big)_{\mc H_T^*} +  \Big( \partial_x f(u), 
         (-\partial_{xx})^s u
          \Big)_{\mc H_T^*}
\end{split}
\end{equation}
Now note that, by the same arguments as in Lemma~\ref{l:ext1},
integration by parts are allowed and
\begin{equation}
\label{e:orth}
\begin{split}
\Big( \partial_x f(u), (-\partial_{xx})^s u\Big)_{\mc H_T^*} =& 
 \Big( \partial_x f(u), 
         u -m\Big)_{L_2([0,T]\times \bb T)} 
 \\ = & \int_0^T \int_{\bb T} \partial_x q(u(t,x))\,dx\,dt =0
\end{split}
\end{equation}
where $q\in C^1(\bb R)$ is such that $q'(v)=(v-m)f'(v)$. On the other
hand
\begin{equation}
\label{e:orth2}
\begin{split}
\Big( \partial_t u, 
         (-\partial_{xx})^s u
          \Big)_{\mc H_T^*} = & \Big( \partial_t u, 
         u -m\Big)_{L_2([0,T]\times \bb T)}  
\\= &
\int_0^T \int_{\bb T} (u(t,x)-m) \partial_t u(t,x)\,dx\,dt
\\  = &
 \frac 12 \int_0^T \int_{\bb T} \partial_t (u(t,x)-m)^2 \,dx\,dt 
               = \frac 12 \|w-m\|_{L_2(\bb T)}^2
\end{split}
\end{equation}
Patching \eqref{e:dec}, \eqref{e:orth}, \eqref{e:orth2} together, the
result follows.
\end{proof}

\begin{proof}[Proof of Lemma~\ref{l:lbve}]
  Let $v:[0,\infty[\times \bb T \to \bb R$ be the solution to
  \eqref{e:1.2} with initial datum $v(0,x)=w(-x)$, and for
  $T_1,\,T_2>0$ let $u \in \mc X_{T_1+T_2}$ be defined as
\begin{equation*}
u(t,x)=
\begin{cases}
(1-\frac t{T_1})m + \frac t{T_1} v(T_2,-x)  & \text{for $t\in [0,T_1]$}
\\
v(T_1+T_2-t,-x) & \text{for $t\in [T_1,T_1+T_2]$}
\end{cases}
\end{equation*}
Note that $u(0,x)=m$ and $u(T,x)=w(x)$ for $T=T_1+T_2$, so that we are
left with the proof of \eqref{e:resto}.

Since $u$ satisfies $\partial_t u+\partial_x f(u)-\frac{\eps}2
(-\partial_{xx})^s u=0$ for $t\in [T_1,T_1+T_2]$, while calculations
are explicit for $t\in [0,T_1]$
\begin{equation}
\label{e:spaccato}
\begin{split}
&
\frac{\eps^{-1}}{2} \big\|\partial_t u + \partial_x f(u)
         - \frac{\eps}2(-\partial_{xx})^s u
          \big\|_{\mc H_{T_1+T_2}^*}^2
\\
&
\qquad =\frac{\eps^{-1}}{2} \big\|\partial_t u + \partial_x f(u)
         - \frac{\eps}2(-\partial_{xx})^s u
          \big\|_{\mc H_{T_1}^*}^2 
\\
&
\qquad \le  \frac{3 \eps^{-1}}{2} \big[\big\|\partial_t u
          \big\|_{\mc H_{T_1}^*}^2 
         + \big\| \partial_x f(u)
          \big\|_{\mc H_{T_1}^*}^2 
         + \big\|\frac{\eps}2(-\partial_{xx})^s u
          \big\|_{\mc H_{T_1}^*}^2 
\big]
\\ 
& \qquad \le
 \frac{3\eps^{-1}}2 \big\|v(T_2)-m \big\|_{\dot H^{-s}(\bb T)}^2
 + \frac{3\eps^{-1}}2 \big\| f'(u) \partial_x u
          \big\|_{\mc H_{T_1}^*}^2 
 + \frac{3\eps\,T_1}{16} \big\|v(T_2)\big\|_{\dot H^{s}(\bb T)}^2
\end{split}
\end{equation}
Now note that if $\omega_1,\,\omega_2 \in \dot H^{-s}(\bb T)$ are such
that
\begin{equation*}
\int_{\bb T} \omega_1(x)\,dx=0 
       \qquad 
\int_{\bb T} \omega_1(x)\,\omega_2(x) dx=0
\end{equation*}
then
\begin{equation}
\label{e:emb1}
\|\omega_1\|_{\dot H^{-s}(\bb T)}^2 \le \|\omega_1\|_{\dot H^{s}(\bb T)}^2
\end{equation}
\begin{equation}
\label{e:emb2}
\|\omega_1\,\omega_2\|_{\dot H^{-s}(\bb T)}^2 
    \le \|\omega_1\|_{\dot H^{-s}(\bb T)}^2
       \,\|\omega_2\|_{L_{\infty}(\bb T)}^2
\end{equation}
Applying \eqref{e:emb1} to the first term of the last line of
\eqref{e:spaccato}; applying \eqref{e:emb2} integrated over $t\in
[0,T_1]$ to the second term in the last line of \eqref{e:spaccato}
with $\omega_1=\partial_x u(t)$ and $\omega_2=f'(u(t))$, one gets
\begin{equation}
\begin{split}
\label{e:spaccato2}
& \frac{\eps^{-1}}{2} \big\|\partial_t u + \partial_x f(u)
         - \frac{\eps}2(-\partial_{xx})^s u
          \big\|_{\mc H_{T_1+T_2}^*}^2
 \\ & \qquad  \le C\,\big( \|\partial_x v(T_2)\|_{\dot H^{-s}(\bb T)}^2 
           +  \|v(T_2)\|_{\dot H^{s}(\bb T)}^2\big)
              \le 2\,C\,  \|v(T_2)\|_{\dot H^{s}(\bb T)}^2
\end{split}
\end{equation}
where $C$ is a constant depending only on $\eps,\,T_1,\,f$, and we
used $\|\partial_x v(T_2)\|_{\dot H^{-s}} \le \| v(T_2)\|_{\dot
  H^{s}}$ as $s>1/2$. Now note that by a standard parabolic estimate
(indeed by \eqref{e:bound1} calculated for $I_{\eps,T}(u)=0$)
\begin{equation*}
     \int_0^{+\infty} \|v(t)\|_{\dot H^{s}(\bb T)}^2 
\le \|w\|_{L_2(\bb T)}^2<+\infty
\end{equation*}
Therefore for each $\gamma>0$, there exists $T_2$ large enough such
that the rightest hand side of \eqref{e:spaccato2} is smaller than
$\gamma$.
\end{proof}

 \textit{Acknowledgment.} The author M.M.  acknowledges ANR "SHEPI", grant ANR-2010-BLAN-0108. The author Y. S. acknowledges the ANR "PREFERED".

\end{document}